\documentclass[12pt]{amsart}
\usepackage{mathtools}
\mathtoolsset{showonlyrefs,showmanualtags} 

\usepackage[top=1.5in, bottom=1.5in, left=1.25in, right=1.25in]{geometry}
\usepackage[all]{xy}
\usepackage{amsmath}
\usepackage{amssymb}
\usepackage{amsthm}
\usepackage{amscd}

\newcommand{\EE}{\mathbb  E}

\newcommand{\Z}{\mathbb  Z}

\newcommand{\C}{\mathbb  C}

\numberwithin{equation}{section}

\newtheorem{theorem}[equation]{Theorem}
\newtheorem{definition}[equation]{Definition}
\newtheorem{proposition}[equation]{Proposition}

\newtheorem{lemma}[equation]{Lemma}

\newtheorem{remark}[equation]{Remark}

\DeclareMathOperator{\N}{\mathbb{N}}

\usepackage{hyperref} 
\hypersetup{
    colorlinks=true,       
    linkcolor=blue,          
    citecolor=magenta,        
    filecolor=magenta,      
    urlcolor=cyan           
}

\begin{document}
\title[A Uniform Random Wiener-Wintner Theorem]{A Uniform Random Pointwise Ergodic Theorem}
\author{Ben Krause}
\address{
Department of Mathematics
The University of British Columbia \\
1984 Mathematics Road
Vancouver, B.C.
Canada V6T 1Z2}
\email{benkrause@math.ubc.ca}
\author{Pavel Zorin-Kranich}
\address{Universit\"at Bonn\\
  Mathematisches Institut\\
  Endenicher Allee 60\\
  53115 Bonn\\
  Germany
}
\email{pzorin@uni-bonn.de}

\date{\today}
\maketitle

\begin{abstract}
Let $a_n$ be the random increasing sequence of natural numbers which takes each
value independently with decreasing probability of order $n^{-\alpha}$, $0 <
\alpha < 1/2$.
We prove that, almost surely, for every measure-preserving system $(X,T)$ and
every $f \in L^1(X)$ orthogonal to the invariant factor the modulated, random
averages
\[\sup_{b} \Big| \frac{1}{N} \sum_{n = 1}^N b(n) T^{a_{n}} f \Big|
\]
converge to $0$ pointwise almost everywhere, where the supremum is taken over a
set of bounded functions with certain uniform approximation properties; good examples of such functions are given by
\[ \mathcal{A}_{\delta,M} := \{ e^{2\pi i t^c} : m + \delta \leq c \leq m+1 - \delta, \ 1 \leq m \leq M \} \]
where $M \geq 1$ and $1/2 \geq \delta > 0$ are arbitrary.
This work improves upon previous work of the authors, in which a non-uniform
statement was proven for some specific functions $b$. Under further conditions on $\{b\}$ we prove pointwise convergence to zero of the above averages for general $f \in L^1(X)$; these conditions are met, for instance, by the sets of functions $\mathcal{A}_{\delta,M}$. \end{abstract}

\section{Introduction}
Pointwise ergodic theory concerns the asymptotic pointwise behavior of the
averages
\begin{equation}\label{avg} \frac{1}{N} \sum_{n \leq N} T^{a_n} f
\end{equation}
where $\{a_n \} \subset \N$, and $f$ is some $L^p$ function in a
\emph{measure-preserving system}: a probability space $(X,\mu)$ equipped with a
measure-preserving transformation, $T : X \to X$. Birkhoff's pointwise theorem
\cite{BI} is simply the statement that, when $a_n = n$, the averages in
\eqref{avg} converge pointwise $\mu$-a.e.\ for $f \in L^1(X)$.

In his celebrated paper \cite{B0}, Bourgain initiated a study of ``random''
pointwise ergodic theorems, where the subsequence $\{ a_n\}$ is randomly
generated.

From now on $\{X_n\}$ will denote a sequence of independent $\{0,1\}$ valued
random variables (on a probability space $\Omega$) with expectations
$\sigma_n$.
The \emph{counting function} $a_n(\omega)$ is the smallest integer subject to
the constraint
\[ X_1(\omega) + \dots + X_{a_n(\omega)}(\omega) = n.\]

Bourgain established the following result.
\begin{theorem}[{\cite[Proposition 8.2]{B0}}]
Suppose
\[
\sigma_n = \frac{ (\log \log n)^{B_p}}{n}, \ B_p > \frac{1}{p-1}, \ 1 < p \leq
2.
\]
Then, almost surely, for each measure-preserving system, and each $f \in
L^p(X)$, the averages $\frac{1}{N} \sum_{n\leq N} T^{a_n} f$ converge pointwise
$\mu$-a.e.
\end{theorem}

Our main result is in the spirit of Bourgain's uniform version \cite{BWW} of
the Wiener--Wintner theorem \cite{WW} (see Assani \cite{AWW} for more results
in this direction).
It is uniform over the following classes of weights.

\begin{definition}
A collection of functions $\mathcal{B} := \{ b\}$, $b: \N \to \C$ is
\emph{approximable} if for every $\delta>0$ there exists some $\kappa = \kappa(\delta) > 0$ so
that for every (sufficiently large) integer $N$, there exist finite subsets
$\mathcal{B}_N \subset \mathcal{B}$ with the following two properties:
\begin{itemize}
\item $|\mathcal{B}_N| \leq C_\delta e^{N^\delta}$ for some constant $C_\delta$ depending only on $\delta$;
\item for any $b \in \mathcal{B}$, there exists some $b_0 \in \mathcal{B}_N$ so
that
\[ \sup_{t \leq N} |b(t) - b_0(t)| \leq C_\kappa N^{-\kappa}.\]
for some $C_\kappa = C_{\kappa(\delta)}$.
\end{itemize}
\end{definition}

Good examples of approximable sets of functions are
\[
\mathcal{B}_I := \{ e(n^c) : c \in I \},
\quad e(t):= e^{2\pi i t},
\]
for finite intervals $I \subset \mathbb{R}$; note that finite unions of approximable
sets remain approximable.

\begin{theorem}\label{Main}
Let $\mathcal{B}$ be an
approximable set of functions bounded in magnitude by $1$, and suppose $\sigma_n = n^{-\alpha}$ for some $0 < \alpha < 1/2$.
Then, almost surely, the following holds:
For every measure-preserving system $(X,\mu,T)$, and every $f \in
L^1(X)$ orthogonal to the invariant factor,
\[
\sup_{b \in \mathcal{B}} \Big| \frac{1}{N} \sum_{n\leq N} b(n) T^{a_n(\omega)}f
\Big| \to 0
\]
$\mu$-a.e.
\end{theorem}
\begin{remark}
The restriction to the orthogonal complement of the invariant factor in Theorem~\ref{Main} can be
removed provided that
\[
\sup_{b \in \mathcal{B}} \Big| \frac{1}{N} \sum_{n\leq N} b(n) \Big| \to 0.
\]
The latter property holds for large classes of \emph{Hardy field} functions, see
\cite[Theorem 2.10]{EK} (with $f = \mathbf{1}_X$); for an introduction to Hardy field functions and their properties, we refer the reader to e.g.\ \cite[\S 2]{EK}.
\end{remark}

Note that by the strong law of large numbers, if $\sigma_n = n^{-\alpha}$, almost surely there exists a constant $C_\omega$ so that for all large $n$, 
\begin{equation}\label{e:growth} C_\omega^{-1} n^{\frac{1}{1-\alpha}} \leq a_n(\omega) \leq C_\omega n^{\frac{1}{1-\alpha}}.
\end{equation}
In particular, our sequence $\{ a_n(\omega) \}$ is asymptotically much denser than the sequence of squares.

This restriction appears in our proof
because we exploit cancellation via a $TT^{*}$ argument.
If $\sigma_{n} = n^{-1/2}$, then almost surely for large $N$, $a_n$ grows like $n^2$, and
fewer than a constant multiple of $N^{1/2}$ elements of the interval $\{1,\dotsc,N\}$ generically appear in the
sequence $\{ a_{n} \}$, so that their difference set fails to cover
$\{1,\dotsc,N\}$ with a (generic) multiplicity which grows with $N$. But, the $TT^*$ argument we use is effective only when the generic behavior of
\[
\sum_{1 \leq n, n+h \leq N} X_{n+h} X_n, \ 1 \leq |h| \leq N 
\]
concentrates strongly around its expected value, which leads to significant cancellation in the centered variant of the above random sum. This concentration occurs whenever $\sigma_n = n^{-\alpha}, \ 0 < \alpha < 1/2$, and LaVictoire's maximal ergodic theorem \cite{LV} similarly exploits this concentration to show that the
corresponding maximal operator has weak type $(1,1)$ under similar conditions.\\

The structure of this paper is as follows:\\
In \textsection\ref{sec:preliminaries} we introduce a few preliminary tools;\\
In \textsection\ref{sec:proposition} we establish our key analytic inequality; \\
Finally, we complete the proof of Theorem \ref{Main} in \textsection \ref{sec:Proof}.

\subsection{Acknowledgments}
We are grateful to Nikos Frantzikinakis for his support, encouragement, and insight. The first author is partially supported by an NSF postdoctoral fellowship.

\section{Preliminaries}\label{sec:preliminaries}
\subsection{Notation and Tools}
With $X_n, \ \sigma_n$ as above, we let $Y_n := X_n - \sigma_n$. We let 
\[ S_{N} =\sum_{n=1}^{N} X_{N} \ \text{ and } \ W_N=\sum_{n=1}^{N} \sigma_{n}\]
so that $W_N$ grows like $N^{1-\alpha}$.

We will make use of the modified Vinogradov notation. We use $X \lesssim Y$, or $Y \gtrsim X$ to denote the estimate $X \leq CY$ for an absolute constant $C$. We use $X \approx Y$ to mean that both $X \lesssim Y$ and $Y \lesssim X$. If we need $C$ to depend on a parameter,
we shall indicate this by subscripts, thus for instance $X \lesssim_\omega Y$ denotes
the estimate $X \leq C_\omega Y$ for some $C_\omega$ depending on $\omega$.

We will require the following large deviation ``martingale" inequality:

Let $\{Z_1,\dots,Z_n\}$ be a sequence of scalar random variables with $|Z_i| \leq 1$ almost surely. Assume also that we have the martingale difference property
\[ \mathbb{E} ( Z_i | Z_1,\dots,Z_{i-1} ) = 0 \]
almost surely for all $1 \leq i \leq n$.
Set 
\[ V_i := \text{Var} ( Z_i | Z_1,\dots,Z_{i-1} ),\] and $T_j := \sum_{i=1}^j V_i$. 
Note that in the case where the $\{Z_i\}$ are independent, $V_i = \mathbb{E} |Z_i|^2$.

Then, we have the following large deviation inequality due to Freedman \cite[Theorem 1.6]{F}.
\begin{proposition}[Freedman's Martingale Inequality, Special Case]\label{FR}
With the above notation, for any real numbers $a,b > 0$,
\[ \mathbb{P}\left( \left| \sum_{i=1}^n Z_i \right| \geq a, T_n \leq b \right) \leq 2 e^{ - \frac{a^2}{2(a+b)} }. \]
\end{proposition}

\section{A Key Proposition}\label{sec:proposition}
The focus of this section is to prove an $\ell^2(\Z)$ inequality for functions $f : \Z \to \mathbb{C}$. Here is the set-up.

Fix some constant of lacunarity $\rho > 1$, which we will think of as arbitrarily close to $1$; henceforth, all upper case indices, $M,N$, etc.\ will belong to the sequence
\[ \{ \lfloor \rho^k \rfloor : k \geq 0\}.\]

For each $N$, suppose that $B_N$ are a finite collection functions, all bounded in magnitude by $1$, with 
\[ |B_N| \lesssim_\delta e^{N^{\delta}}\] for any $\delta > 0$.

We will be interested in bounding the $\ell^2$-norm of the maximal functions
\[ \mathcal{M}_N f(x) := \sup_{b \in B_N} \left| \frac{1}{N^{1-\alpha}} \sum_{n \leq N} Y_n b(S_{n-1}) f(x-n) \right| \]
with high probability. Here is our proposition.

\begin{proposition}\label{KEY}
Suppose $0 < \alpha < 1/2$. Then, for some $\epsilon = \epsilon(\alpha) > 0$, $\omega$-almost surely we may estimate
\[ \| \mathcal{M}_N f \|_{\ell^2} \lesssim_\omega N^{-\epsilon} \|f \|_{\ell^2}.\]
\end{proposition}
\begin{proof}
The proof is by linearization and $TT^*$. Specifically, for an appropriate disjoint partition of $\Z$, $\{ E_b \}_{b \in B_N}$, we may express
\[ \mathcal{M}_N f = |T_Nf |, \]
where
\[ T_Nf(x) := \frac{1}{N^{1-\alpha}} \sum_{b \in B_N} \mathbf{1}_{E_b}(x) \sum_{n \leq N} Y_n b(S_{n-1}) f(x-n). \]
Then, $T_NT_N^*f(x)$ can be expressed as the sum of two terms
\begin{equation}\label{main term} \frac{1}{N^{2-2\alpha}} \sum_{b,b'} \mathbf{1}_{E_b}(x) \sum_{0 < |h| \leq N} K_N(h;b,b') (\mathbf{1}_{E_{b'}}f)(x+h),\end{equation}
where
\begin{equation}\label{K}
K_N(h; b, b') := \sum_{1\leq n, n+h \leq N} Y_{n+h} Y_n b(S_{n+h-1}) \overline{b'}(S_{n-1})
\end{equation}
and
\begin{equation}\label{simple term}
\left(\frac{1}{N^{2-2\alpha}} \sum_{n \leq N} Y_n^2\right) \cdot f(x). \end{equation}
The goal will now be to show that, $\omega$-almost surely
\begin{equation}\label{goal} \| T_NT_N^* f\|_{\ell^2} \lesssim_\omega N^{-2 \epsilon} \|f\|_{\ell^2}.
\end{equation}
Now, by Proposition \ref{FR}, or more simply by Chernoff's inequality, \cite{TV}, and a Borel-Cantelli argument, we see that $\omega$-almost surely
\[ \| \eqref{simple term} \|_{\ell^2} \lesssim_\omega N^{\alpha - 1} \|f\|_{\ell^2},\]
so we will disregard it in what follows. We will also restrict attention in what follows to positive $1 \leq h \leq N$, as the case of negative $h$ can be handled by similar arguments.

We begin with the following observation, which we state in the form of the following lemma. 
\begin{lemma}\label{Tech}
For any $\delta > 0$, there exists an absolute constant $c$ so that
\[ \mathbb{P}\left( \sup_{1 \leq h \leq N} \left| \sum_{n=1}^{N-h} \sigma_{n+h} \left( |Y_n|^2 - \mathbb{E}|Y_n|^2 \right) \right| \gtrsim N^{1-2\alpha} \right) \lesssim_\delta e^{ -c N^{1-\alpha - \delta}}. \]
\end{lemma}
\begin{proof}[Sketch]
The trivial union bound allows one to estimate the inner probability without the supremum, for a sub-exponential loss in $N$. The result then follows from Proposition \ref{FR}, or more simply from Chernoff's inequality, \cite{TV}.
\end{proof}

Let us consider the kernel
\[K_N(h; b, b') := \sum_{1\leq n, n+h \leq N} Y_{n+h} Y_n b(S_{n+h-1}) \overline{b'}(S_{n-1});\]
since $Y_{n+h}$ is independent from all other random variables appearing in each summand, $K_N(h;b,b')$ is a sum of martingale increments. Its conditional variance is given by
\[ T_N(h) := \sum_{n=1}^{N-h} \sigma_{n+h}^2 |Y_n|^2.\]
We expand the foregoing out as
\[ \sum_{n=1}^{N-h} \sigma_{n+h} \EE |Y_n|^2 + \sum_{n=1}^{N-h} \sigma_{n+h} \left( |Y_n|^2 - \mathbb{E}|Y_n|^2 \right),\]
which we may bound, in light of the previous technical Lemma \ref{Tech}, by a constant multiple of $N^{1-2\alpha}$ away from a set of probability $\lesssim_\delta e^{-c N^{1-\alpha - \delta}}$.

We will now apply Proposition \ref{FR} to estimate the magnitudes of 
\[ \{ K_N(h;b,b') : 1 \leq |h| \leq N, \ b,b' \in B_N \}.\]
First, choose $\epsilon = \epsilon(\alpha) > 0$ so small that 
\[ 1 - 2\alpha - 10 \epsilon > 0,\]
and bound
\[ \mathbb{P}( |K_N(h;b,b')| \gtrsim N^{1-2\alpha - 2\epsilon} ) \]
by a constant multiple (determined by $\epsilon = \epsilon(\alpha)$) of 
\[ 
\mathbb{P}(|K_N(h;b,b')| \gtrsim N^{1-2\alpha - 2 \epsilon}, |T_N(h)| \lesssim N^{1-2\alpha}) + e^{-c N^{1-\alpha - \epsilon}}; \]
Freedman's Martingale inequality, Proposition \ref{FR}, then allows us to bound the foregoing by a constant multiple of
\[ e^{-c N^{1-2\alpha - 4\epsilon}}.\]
Using the crude union bound, and the cardinality estimate 
\[ |B_N| \lesssim_\epsilon e^{ N^{\epsilon}},\]
we may pass to the estimate
\begin{equation}\label{conclusion0}
\mathbb{P}\left( \sup_{b,b' \in B_N} \sup_{1\leq |h| \leq N} |K_N(h;b,b')| \gtrsim N^{1-2\alpha - 2\epsilon} \right) 
\lesssim e^{-c N^{1-2\alpha - 5\epsilon}}.
\end{equation}

In particular, by a Borel-Cantelli argument, $\omega$-almost surely, 
\begin{equation}\label{conclusion1}
\sup_{b,b' \in B_N} \sup_{1\leq |h| \leq N} |K_N(h;b,b')| \lesssim_\omega N^{1-2\alpha - 2 \epsilon}.
\end{equation}

We are now ready to quickly prove Proposition \ref{KEY}:

Up to \eqref{simple term}, we may almost surely bound
\[ \aligned 
|T_N T_N^* f(x)| &\lesssim_\omega \frac{1}{N^{2-2\alpha}} \sum_{b,b' \in B_N} \mathbf{1}_{E_b}(x) N^{1-2\alpha - 2 \epsilon} \sum_{ 1 \leq |h| \leq N} (\mathbf{1}_{E_{b'}}f)(x+h) \\
&\lesssim_\omega N^{-2\epsilon} M_{HL} f(x),\endaligned\]
where $M_{HL}$ is the standard Hardy-Littlewood maximal function. The result follows.
\end{proof}

With this in mind, we are ready for our proof of Theorem \ref{Main}.

\section{The Proof of Theorem \ref{Main}}\label{sec:Proof}
We begin by using the almost-sure weak-type $1-1$ boundedness of the maximal function
\[ \aligned 
M_{LV}f &:= \sup_N \frac{1}{N} \sum_{n\leq N} T^{a_n(\omega)}|f| \\
&\equiv \sup_{N : S_N \neq 0} \frac{1}{S_N} \sum_{n \leq N} X_n(\omega) T^n |f| \\
&\approx_\omega
\sup_N \frac{1}{W_N} \sum_{n \leq N} X_n(\omega) T^n |f|
\endaligned \]
\cite{LV} to replace general $L^1$ functions appearing in the statement of Theorem \ref{Main} by $f \in \{ h - Th : h \in L^{\infty}(X) \}$; the full strength of Theorem \ref{Main} may be recovered by a standard density argument. Note the usage of the strong law of large numbers to conclude that, $\omega$-almost surely
\[ S_n \approx_\omega W_n \]
for all $n$ such that $S_n \neq 0$.

Now, by the boundedness of $f$ and Rosenblatt, Wierdl \cite[Lemma 1.5]{RW}, it is enough to restrict attention to lacunary sequences $N \in \{ \lfloor \rho^{k} \rfloor : k\geq 0\}$, where $\rho$ is taken from a countable sequence converging to $1$. We will fix some $\rho > 1$ throughout, and the averaging parameters are assumed to belong to $\{ \lfloor \rho^k \rfloor : k \geq 0\}$. 

By definition, it is enough to prove the stated convergence for
\[ \sup_{b \in \mathcal{B}} \left| \frac{1}{S_N} \sum_{n \leq N} X_n b(S_n)  T^{n} f \right|;\]
since 
\[ X_n b(S_n) = X_n b(S_{n-1} + 1)\]
for any function $b$, it is enough to prove pointwise convergence to zero (along lacunary times) for
\begin{equation}\label{max1}
 \sup_{b \in \mathcal{B}} \left| \frac{1}{S_N} \sum_{n \leq N} X_n b(S_{n-1}+1) T^{n} f \right|.
\end{equation}
By the strong law of large numbers (or by an easy application of Chernoff's inequality \cite{TV}), we know that almost surely $\frac{S_N}{W_N} \to 1$; consequently we may instead prove our convergence result for
\[ \sup_{p \in \mathcal{B}} \left| \frac{1}{W_N} \sum_{n \leq N} X_n b(S_{n-1}+1) T^{n} f \right|.\]

By the definition of approximability, for any $b \in \mathcal{B}$ and $N \in \lfloor \rho^{\mathbb{N}} \rfloor$, there exists a $b_0 \in \mathcal{B}_N \subset \mathcal{B}$ so that
\[ \aligned
&\left| \frac{1}{W_N} \sum_{n \leq N} X_n b(S_{n-1}+1) T^{n} f - 
\frac{1}{W_N} \sum_{n \leq N} X_n b_0(S_{n-1}+1) T^{n} f \right| \\
& \qquad \lesssim
\frac{1}{W_N} \sum_{n\leq N} X_n \cdot \left( \sup_{t \leq N} |b(t) - b_0(t)| \right) \cdot T^n |f| \\
& \qquad \qquad \lesssim_\omega
N^{-\kappa} M_{LV} f, \endaligned\]
for some $\kappa > 0$.
In particular, almost surely, for each $N$ we may bound
\[ \sup_{b \in \mathcal{B}} \left| \frac{1}{W_N} \sum_{n \leq N} X_n b(S_{n-1}+1) T^{n} f \right| \lesssim_\omega 
\sup_{b \in \mathcal{B}_N} \left| \frac{1}{W_N} \sum_{n \leq N} X_n b(S_{n-1}+1) T^{n} f \right| + N^{-\kappa} M_{LV} f;\]
since almost surely this latter functions tends to zero pointwise $\mu$-a.e. it suffices to consider the first term on the right. Since $W_N = c_\alpha N^{1-\alpha} + O(1)$, we may replace the normalizing factor $W_N$ with $N^{1-\alpha}$. At this point, we have reduced the problem to showing that, under the above hypotheses, $\omega$-almost surely, for any measure-preserving system $(X,\mu,T)$,
\begin{equation}\label{goal}
\lim_{N\to\infty, N\in \lfloor \rho^{\mathbb{N}}\rfloor} \sup_{b \in \mathcal{B}_N} \left| \frac{1}{N^{1-\alpha}} \sum_{n \leq N} X_n b(S_{n-1}+1) T^{n} f \right| = 0
\end{equation}
$\mu$-a.e. for each $f \in L^1(X)\cap L^\infty(X)$.

By Proposition \ref{KEY}, Calder\'{o}n's transference principle \cite{C}, and a Borel-Cantelli argument, we know that, almost surely,
\[
\sup_{b \in \mathcal{B}_N} \left| \frac{1}{N^{1-\alpha}} \sum_{n \leq N} Y_n b(S_{n-1}+1) T^{n} f \right| \to 0
\]
$\mu$-a.e. along lacunary times. This is since, $\omega$-almost surely,
\[ \aligned
&\mu\left( \limsup_N \sup_{b \in \mathcal{B}_N} \left| \frac{1}{N^{1-\alpha}} \sum_{n \leq N} Y_n b(S_{n-1}+1) T^{n} f \right| \geq t \right) \\
& \qquad \leq \lim_M
\sum_{N \geq M} \mu\left( \sup_{b \in \mathcal{B}_N} \left| \frac{1}{N^{1-\alpha}} \sum_{n \leq N} Y_n b(S_{n-1}+1) T^{n} f \right| \geq t \right) \\
& \qquad \qquad \lesssim_\omega \lim_M \sum_{N \geq M} t^{-2} N^{-2\epsilon} \| f\|_{L^2(X)}^2 \\
&\qquad \qquad \qquad \lesssim \lim_M M^{-2\epsilon} t^{-2}\| f\|_{L^2(X)}^2 = 0; \endaligned\]
we were able to estimate $\sum_{N \geq M } N^{-2\epsilon} \lesssim M^{-2\epsilon}$ since our averaging parameters are restricted to a lacunary sequence.
The upshot is that we have reduced matters to proving the following lemma. 
\begin{lemma}\label{final}
$\omega$-almost surely, for any measure-preserving system, and each simple $f = h-Th$, $h\in L^\infty(X)$
\begin{equation}\label{max3}
 \sup_{b \in \mathcal{B}_N} \left| \frac{1}{N^{1-\alpha}} \sum_{n \leq N} \sigma_n b(S_{n-1}+1) T^{n} f \right| \to 0
\end{equation}
$\mu$-a.e.
\end{lemma}
\begin{proof}[Proof of Lemma \ref{final}]
Substituting $f = h - Th$ and summing by parts lets us bound the above maximal function by a constant multiple of 
\[  \frac{\|h\|_\infty}{N^{1-\alpha}} + \sup_{b \in \mathcal{B}} \left| \frac{1}{N^{1-\alpha}} \sum_{2 \leq m\leq N} \left( \sigma_m b(S_{m-1}+1) - \sigma_{m-1} b(S_{m-2}+1) \right) T^m h \right|,\]
which is in turn bounded by a constant multiple of
\[  \frac{\|h\|_\infty}{N^{1-\alpha}} + \frac{1}{N^{1-\alpha}} \sum_{2 \leq m \leq N} (m-1)^{-\alpha -1} \|h\|_\infty +
\frac{1}{N^{1-\alpha}} \sum_{2 \leq m \leq N} (m-1)^{-\alpha } X_{m-1} \|h\|_\infty.\]
The first two terms in the sum clearly tend to zero as $N \to \infty$. For the third term, we proceed as follows.
By the strong law of large numbers, $\omega$-almost surely, for any dyadic $2 \leq K \leq N$ 
\[ \sum_{K/2 < m \leq K} (m-1)^{-\alpha} X_{m-1} \lesssim K^{-\alpha} S_K \lesssim_\omega K^{1-2\alpha},\]
and so we may almost surely bound
\[ \frac{1}{N^{1-\alpha}} \sum_{2 \leq m \leq N} (m-1)^{-\alpha } X_{m-1} \|h\|_\infty \lesssim N^{-\alpha } \|h\|_\infty \to 0 \]
as well.
This completes the proof of Lemma \ref{final}, and with it, the proof of Theorem \ref{Main}.
\end{proof}

\end{document}